\documentclass[11pt, leqno]{amsart}

\usepackage{
  amsmath,
  amsthm,
  amssymb,
  tikz,
  mathdots,
  fancyhdr,
  enumerate,
  multirow,
  mathtools
}

\usetikzlibrary{arrows,calc,matrix,shapes}

\usepackage[
  colorlinks=true,
  linkcolor=black,
  anchorcolor=black,
  citecolor=black,
  menucolor=black,
  filecolor=black,
  urlcolor=black
] {hyperref}

\usepackage[center,small,sc]{caption}

\usepackage{multicol}

\usepackage[all]{xy}

\usepackage[colorinlistoftodos]{todonotes}
\theoremstyle{plain}
\newtheorem{thm}{Theorem}[section]
\newtheorem{lemma}[thm]{Lemma}
\newtheorem{prop}[thm]{Proposition}
\newtheorem{cor}[thm]{Corollary}

\theoremstyle{definition}
\newtheorem{dfn}[thm]{Definition}
\newtheorem{ex}[thm]{Example}

\newtheorem{remark}[thm]{Remark}
\numberwithin{equation}{section} \numberwithin{figure}{section}
\numberwithin{table}{section}

\newcommand*\circled[1]{\tikz[baseline=(char.base)]{
            \node[shape=circle,draw,inner sep=2pt,scale=.6] (char) {#1};}}
\newcommand{\matr}[2]{\left( \begin{matrix} #1 \\ #2 \end{matrix} \right)}

\newcommand{\la}{\langle}
\newcommand{\ra}{\rangle}

\newcommand{\bw}{\boldsymbol{w}}
\newcommand{\coll}[2]{{\bf c}^{#1}_{#2}}
\newcommand{\pw}{\bw_0}
\newcommand{\sw}{\bw'}

\newcommand{\F}{\mathfrak{F}}

\newcommand{\bv}{\boldsymbol v}

\newcommand{\ts}{\tilde s}

\newcommand{\PP}{\mathcal{P}}

\newcommand{\GS}{Gr{\" o}bner--Shirshov }

\newenvironment{red}{\relax\color{red}}{\relax}
\newenvironment{blue}{\relax\color{blue}}{\hspace*{.5ex}\relax}

\newcommand{\ber}{\begin{red}}
\newcommand{\er}{\end{red}}
\newcommand{\beb}{\begin{blue}}
\newcommand{\eb}{\end{blue}}

\def\ncnode#1#2{{\kern -1pt\mathop\bigcirc\limits_{#2}
                \kern-11pt{\scriptstyle#1}\kern 4pt}}
 \def\nRnode#1#2{{\kern -0.4pt\mathop\Box\limits_{#2}
   \kern-8.6pt{\scriptstyle#1}\kern 2.3pt}}
\def\sbar#1pt{{\vrule width#1pt height3pt depth-2pt}}
\def\dbar#1pt{{\rlap{\vrule width#1pt height2pt depth-1pt}
                 \vrule width#1pt height4pt depth-3pt}}

\tikzset{tab/.style={matrix of math nodes,column sep=-.4, row
sep=-.4,text height=8pt,text width=8pt,align=center}}

\begin{document}

\title[Fully commutative elements]{Fully commutative elements of \\ the complex reflection groups}

\author[G. Feinberg, S. Kim, K.-H. Lee, S.-j. Oh]{Gabriel Feinberg, Sungsoon Kim$^{\diamond}$, Kyu-Hwan Lee$^{\star}$ and Se-jin Oh$^\dagger$}

\thanks{$^{\star}$This work was partially supported by a grant from the Simons Foundation (\#318706).}

\thanks{$^{\diamond}$This author is grateful to KIAS (Seoul, South Korea) for its hospitality during this work.}

\thanks{$^{\dagger}$This work was partially supported by the National Research Foundation of Korea(NRF) Grant funded by the Korea government(MSIP) (NRF-2016R1C1B2013135)}

\address{Department of Mathematics and Computer Science,
Washington College, Chestertown, MD 21620, U.S.A. }
\email{gfeinberg2@washcoll.edu}

\address{D{\' e}partement de Math{\' e}matiques, LAMFA,
Universit{\' e} de Picardie Jules-Verne,
33 rue St. Leu,
80039 Amiens,
France}
\email{sungsoon.kim@u-picardie.fr}

\address{Department of Mathematics, University of Connecticut, Storrs, CT 06269, U.S.A.}
\email{khlee@math.uconn.edu}

\address{Department of Mathematics Ewha Womans University Seoul 120-750, South Korea}
\email{sejin092@gmail.com}

\date{\today}

\begin{abstract}

We extend the usual notion of fully commutative elements from the Coxeter groups to the complex reflection groups. Then we decompose the sets of fully commutative elements into natural subsets according to their combinatorial properties, and investigate the structure of these decompositions. As a consequence, we enumerate and describe the form of these elements in the complex reflection groups.
\end{abstract}

\maketitle

\section{Introduction}
An element $w$ of a Coxeter group is said to
be {\em fully commutative} if any reduced word for $w$ can be
obtained from any other  by  interchanges of adjacent commuting
generators.
In \cite{Stem96},  Stembridge classified the Coxeter
groups that have finitely many fully commutative elements. His
results completed the work of Fan \cite{Fan1996} and Graham \cite{Graham} who had obtained such a classification
for the simply-laced types and had
shown that the fully commutative elements parameterize natural
bases for the corresponding quotients of Hecke algebras.  In type $A_n$,
the quotients are isomorphic to the Temperley--Lieb algebras (see
\cite{Jones1987}). Fan and Stembridge also enumerated the set of
fully commutative elements.  In particular, they showed the
following.

\begin{prop}[\cite{Fan1996,Stem98}]\label{full comm count}  Let $C_n$ be the $n^\textrm{th}$ Catalan number, i.e. $C_n = \frac{1}{n+1}{2n\choose n}$.
Then the numbers of fully commutative elements in the Coxeter groups
of types $A_n$, $B_n$  and $D_n$ are given as follows:
$$
\begin{cases}
C_{n+1} & \text{ if the type is } A_n,
\\
(n+2) \,  C_{n}-1  & \text{ if the type is } B_n,
\\ \frac{n+3}{2} \, C_{n}-1 & \text{ if the type is } D_n.
\end{cases}
$$
\end{prop}

This paper is concerned with the complex reflection groups $G(d,r,n)$, where $d,r, n \in \mathbb Z_{>0}$ such that $r|d$. These groups are  generated by complex reflections and have the Coxeter groups of types $A_{n-1}$, $B_n$ and $D_n$ as special cases. As the complex reflection groups can be presented by analogues of simple reflections and braid relations, one can attempt to generalize the notion of full commutativity to these groups. However, a direct generalization using the usual set of braid relations does not work even for $G(d,1,n)$ if $d\ge 3$. A breakdown comes from the fact that some reduced words may not be connected to others strictly using braid relations.

In this paper, we overcome the difficulty and define fully commutative elements for $G(d,1,n)$, by proving that a slightly extended set of braid relations connects all the reduced words for an element of $G(d,1,n)$. (See Example \ref{ex-over}.) The next task is  to describe and enumerate all the fully  commutative elements, and we take the approach of the paper \cite{FL14} where the first and third named authors studied fully commutative elements of the Coxeter group of type $D_n$.

More precisely, we decompose the set of fully commutative elements into natural subsets, called {\em collections}, according to  their canonical words, and group them together into {\em packets} $\PP(n,k)$, $0 \le k \le n$,  so that all the collections in a packet have the same cardinality. We show that the number of fully commutative elements in any collection belonging to the packet $\PP(n,k)$ is equal to the Catalan triangle number $C(n,k)$.
Then the total number of fully commutative elements in $G(d,1,n)$ can be written as
\begin{equation}  \label{eq-11}
\sum_{k=0}^n C(n,k) \left | \PP(n, k) \right | = d(d-1)\F_{n,n-2}(d) +(2d-1)C_n -(d-1),
\end{equation}
where  $|\PP(n,k)|$ is the number of collections in the $(n,k)$-packet and $\F_{n,k}(x)$ is the Catalan triangle polynomial defined by
\begin{align} \label{eq: n,k Cat tr pol}
\F_{n,k}(x) & = \sum_{s=0}^{k} C(n,s)x^{k-s}.
\end{align}
When $d=2$, the group $G(2,1,n)$ is isomorphic to the Coxeter group of type $B_n$, and our definition of fully commutative elements coincides with the usual definition for Coxeter groups, and we recover the known number $(n+2)\, C_n-1$ from \eqref{eq-11}.

Our method for proofs exploits combinatorics of canonical words and establishes bijections among collections. In particular, we realize the Catalan triangle (Table \ref{CT}) using collections of fully commutative elements.

For the group $G(d,r,n)$, $r >1$, we fix an embedding  into $G(d,1,n)$ and define $w \in G(d,r,n)$ to be fully commutative if its image under the embedding is fully commutative in $G(d,1,n)$. The main benefit of this definition is that the decomposition into collections and packets still works without any complications, and we obtain complete description and enumeration of fully commutative elements for all $G(d,r,n)$. On the other hand, a drawback of this definition is that some fully commutative elements in the Coxeter group of type $D_n$ or $G(2,2,n)$ are not fully commutative in $G(2,1,n)$ after being embedded. That is, the usual definition of full commutativity for $D_n$ is not compatible with the new definition.

Though it is not clear at the present, an intrinsic definition of full commutativity for $G(d,r,n)$, $r>1$, which does not use an embedding, may be found. For such a definition, precise information about a complete set (or \GS basis) of relations would be very helpful. We leave it as a future direction.

The organization of this paper is as follows. In Section \ref{first} we determine canonical words for the elements of the complex reflection groups. In the next section, we define fully commutative elements.
Section \ref{ch:D} is devoted to a study of decomposition of the set of fully commutative elements into collections and packets for  $G(d,1,n)$. In Section \ref{sec-gddn}, we consider the packets of $G(d,r,n)$. The next section provides some examples, and the final section is an appendix with the list of reduced words for $G(3,3,3)$.


\section{Canonical forms} \label{first}


\subsection{Complex reflection groups}


For positive integers  $d$ and $n$, let $G(d,1,n)$ be
the group generated by the elements $s_1,s_2
\cdots, s_{n}$ with defining relations:
\begin{subequations}
\begin{align}
s_n^d&=s_i^2=1 \quad &\mbox{for}\ 1\le i \le n-1, \label{eq:order1} \allowdisplaybreaks \\
s_is_j&=s_js_i \quad &\mbox{for}\ j+1 < i \le n, \label{eq:cm1}\allowdisplaybreaks \\
s_{i+1}s_is_{i+1}&=s_is_{i+1}s_i \quad &\mbox{for}\ 1\le i \le n-2, \label{eq:b1} \allowdisplaybreaks \\
s_ns_{n-1}s_ns_{n-1}&=s_{n-1}s_ns_{n-1}s_n. \label{eq:l1}
\end{align}
\end{subequations}
Then the group $G(d,1,n)$
is a complex reflection group which is isomorphic to the wreath product of the cyclic group $\mathbb Z\slash  d \mathbb Z$ and the symmetric group $S_n$.
The corresponding  diagram is given by
$$
\xymatrix@R=0.5ex@C=4ex{ *{\circled{\phantom{$d$}}}<3pt> \ar@{-}[r]_<{ \ 1}  &*{\circled{\phantom{$d$}}}<3pt>
\ar@{-}[r]_<{ \ 2}  &   {} \ar@{.}[r]_>{ \ \ \qquad n-2} & *{\circled{\phantom{$d$}}}<3pt>
\ar@{-}[r]_>{ \ \ \qquad n-1} &*{\circled{\phantom{$d$}}}<3pt>\ar@{=}[r]_>{
\qquad n} &*{\circled{$d$}}<3pt> } $$

For each $d \ge 2$ and $n \ge 3$, let $G(d,d,n)$ be
the complex reflection group generated by the elements $\ts_1,\ts_2
\cdots, \ts_{n}$ with defining relations:
\begin{subequations}
\begin{align}
\ts_i^2&= 1=(\ts_n\ts_{n-1})^d \quad &\mbox{for}\ 1\le i \le n, \allowdisplaybreaks \\
\ts_i\ts_j&=\ts_j\ts_i \quad &\mbox{for}\ j+1 < i \le n-1, \label{eq:cm2}  \allowdisplaybreaks\\
\ts_n\ts_j & =\ts_j\ts_n  & \text{for } j \le n-3, \allowdisplaybreaks \\
\ts_{i+1}\ts_i\ts_{i+1}&=\ts_i\ts_{i+1}\ts_i \quad &\mbox{for}\ 1\le i \le n-2, \label{eq:b2} \allowdisplaybreaks \\
\ts_n \ts_{n-2} \ts_n & = \ts_{n-2} \ts_n \ts_{n-2}, \\ (\ts_n\ts_{n-1}\ts_{n-2})^2 &=(\ts_{n-2}\ts_n\ts_{n-1})^2. \label{eq:l2}
\end{align}
\end{subequations}
The corresponding  diagram is the following.

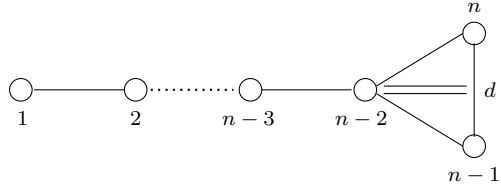
\[
 \begin{tikzpicture}[x=2cm, scale=.5]
    \foreach \x in {-1.5, 0,1.5,3}
    \draw[xshift=\x] (\x,0) circle (3 mm);
    \draw (4.5,1.5) circle (3 mm);
    \draw (4.5,-1.5) circle (3 mm);
    \draw (-1.35,0) --(-0.15,0);
    \draw (1.67,0)--(2.87,0);
    \draw[dotted, thick] (0.2,0) -- (1.3,0);
    \foreach \y in {1.5}
    \draw (3.3,0.1)--(4.4,0.1)
          (3.3,-0.1)--(4.4,-0.1);
    \draw (3.2,0.1)--(4.35,1.5)
          (3.2,-0.1)--(4.35,-1.5);
    \draw (4.5,1.24) -- (4.5,-1.25);

        \draw (4.5,0) node[right]{\scriptsize $d$};
        \draw (-1.5,-.3) node[below]{\scriptsize $1$};
    \draw (0,-.3) node[below]{\scriptsize $2$};
    \draw (1.5,-.3) node[below]{\scriptsize $n-3$};
    \draw (3,-.3) node[below]{\scriptsize $n-2$};
    \draw (4.5,1.8) node[above]{\scriptsize $n$};
    \draw (4.5,-1.8) node[below]{\scriptsize $n-1$};
  \end{tikzpicture}
\]

Note that the complex reflection groups
$G(2,1,n)$ and $G(2,2,n)$ are the Coxeter groups of types $B_n$ and $D_n$
respectively.

For $r\ | \ d$ and $e=d/r$, let $G( d,r ,n)$ be
the complex reflection group generated by the elements $\ts_1,\ts_2
\cdots, \ts_{n}, \ts$ with defining relations:
\begin{subequations}
\begin{align}
\ts^e=\ts_i^2&= 1 \quad &\mbox{for}\ 1\le i \le n, \allowdisplaybreaks\\
\ts_i\ts_j&=\ts_j\ts_i \quad &\mbox{for}\ j+1 < i \le n-1, \label{eq:cm3} \allowdisplaybreaks\\
\ts\ts_j & =\ts_j\ts  & \text{for } j \le n-2, \allowdisplaybreaks\\
\ts_n\ts_j & =\ts_j\ts_n  & \text{for } j \le n-3, \allowdisplaybreaks\\
\ts_{i+1}\ts_i\ts_{i+1}&=\ts_i\ts_{i+1}\ts_i \quad &\mbox{for}\ 1\le i \le n-2, \label{eq:b3} \allowdisplaybreaks\\
\ts \ts_{n} \ts_{n-1} &= \ts_{n} \ts_{n-1} \ts \allowdisplaybreaks\\
\ts_n \ts_{n-2} \ts_n & = \ts_{n-2} \ts_n \ts_{n-2}, \allowdisplaybreaks\\ 
(\ts_n\ts_{n-1}\ts_{n-2})^2 &=(\ts_{n-2}\ts_n\ts_{n-1})^2 \\ \ts \ts_{n} (\ts_{n-1}\ts_{n})^{r-1} &= \ts_{n-1} \ts. \label{eq:l3}
\end{align}
\end{subequations}


\subsection{Canonical reduced words for $G(d,1,n)$}
An  expression $s_{i_1} \cdots
s_{i_r} \in G(d,1,n)$ will be identified with the word $[i_1,  \dots, i_r]$ in  the alphabet $I := \{1, 2, \dots, n \}$.
For $1\leq i ,j \leq n$, we define the words $s_{ij}$ by:
$$ s_{ij} =
 \left\{
   \begin{array}{ll}
     ~[i, i-1, \hdots, j] & \text{ if } i > j, \\
     ~[i] & \text{ if } i = j, \\
     ~[~] & \textrm{ if } i<j ,
   \end{array}
 \right.
$$
where $[~]$ denotes the empty word or the identity element of $G(d,1,n)$.
We will often write $s_{ii}=s_i$. We also define \[ s^{(k)}_{nj} =s_n^k s_{n-1,j} \qquad \text{ for } k \ge 1.\]

The following lemmas are useful to obtain  a
canonical form of the elements of $G(d,1,n)$.

\begin{lemma} \label{lem-GS-Grd1}
The following relations hold in $G(d,1,n)$:
\begin{subequations}
\begin{align}
s_{ij}s_{i}& =s_{i-1}s_{ij} & & \text{ for } j < i \le n-1 , \label{eq:s}  \allowdisplaybreaks \\
s^{(k_1)}_{n,n-1} \, s^{(k_2)}_{n,n-1}& =s_{n-1}s^{(k_2)}_{nj}s^{k_1}_{n} & & \text{ for }  k_1, k_2\ge 1. \label{eq:nt}
\end{align}
\end{subequations}
\end{lemma}

\begin{proof}
Since the relations \eqref{eq:s}  can be checked easily, we only prove the relations \eqref{eq:nt}. We first establish
\begin{equation} \label{eq:tss}
s_{n,n-1} \, s^{(k)}_{n,n-1} =s_{n-1}s^{(k)}_{n,n-1}s_{n}  \qquad \text{ for }  k\ge 1.
\end{equation}
When $k=1$, it is just the defining relation \eqref{eq:l1}. For $k\ge 2$, we use induction on $k$ and obtain
\begin{align*} 
s_{n,n-1} \, s^{(k)}_{n,n-1} &= s_{n,n-1} \, s^{(k-1)}_{n,n-1} s_{n-1}s_ns_{n-1} &(\text{using $s_{n-1}^2=1$)} \allowdisplaybreaks\\ 
&=s_{n-1} s^{(k-1)}_{n,n-1} s_n s_{n-1}s_ns_{n-1}  &(\text{induction)} \allowdisplaybreaks \\ 
&= s_{n-1} s^{(k-1)}_{n,n-1} s_{n-1}s_ns_{n-1}s_n  &(\text{using \eqref{eq:l1})} \allowdisplaybreaks\\ 
&= s_{n-1}s^{(k)}_{n,n-1}s_{n}. 
\end{align*}

Next we prove \[ s^{(k_1)}_{n,n-1} \, s^{(k_2)}_{n,n-1} = s_{n-1} s^{(k_2)}_{n,n-1} s^{k_1}_n \quad \text{ for } k_1,k_2 \ge 1. \] The case $k_1=1$ is obtained above, and when $k_1 \ge 2$, we see
\begin{align*}
s^{(k_1)}_{n,n-1} \, s^{(k_2)}_{n,n-1} &= s_ns^{(k_1-1)}_{n,n-1} \, s^{(k_2)}_{n,n-1} = s_{n,n-1} s^{(k_2)}_{n,n-1} \, s^{k_1-1}_{n}  &(\text{induction)} \allowdisplaybreaks\\ 
&=s_{n-1} s^{(k_2)}_{n,n-1} s^{k_1}_n  &(\text{relation \eqref{eq:tss}}).  
\end{align*}
\end{proof}

\begin{lemma} \label{lem-only}
The following relations hold in $G(d,1,n)$:
\begin{align}
s^{(k_1)}_{nj} \, s^{(k_2)}_{nj}& =s_{n-1}s^{(k_2)}_{nj}s^{(k_1)}_{n, j+1} & \text{ for } j\le n-1 \text{ and } k_1, k_2\ge 1. \label{eq:t}
\end{align}
Moreover, these relations are derived from \eqref{eq:cm1}, \eqref{eq:b1} and \eqref{eq:nt}.
\end{lemma}

\begin{proof}
When $j=n-1$, it is \eqref{eq:nt}. We use downward induction on $j$, and obtain
\begin{align*}
s^{(k_1)}_{nj}\, s^{(k_2)}_{nj} &= s^{(k_1)}_{n, j+1} s^{(k_2)}_{n, j+2} s_j s_{j+1} s_j = s^{(k_1)}_{n, j+1} s^{(k_2)}_{n, j+2} s_{j+1} s_{j} s_{j+1}  & (\text{relations  \eqref{eq:cm1}, \eqref{eq:b1}}) \allowdisplaybreaks \\ 
&= s^{(k_1)}_{n, j+1} s^{(k_2)}_{n, j+1} s_{j} s_{j+1} = s_{n-1} s^{(k_2)}_{n, j+1} s^{(k_1)}_{n,j+2}  s_{j} s_{j+1} & (\text{induction}) \allowdisplaybreaks \\ 
&=s_{n-1} s^{(k_2)}_{n j} s^{(k_1)}_{n,j+1} & (\text{relation  \eqref{eq:cm1}}).
\end{align*}
\end{proof}

Let $\mathcal R$ be the following set of relations:
\begin{subequations}
\begin{align}
s_n^d&=s_i^2=1 & \text{ for }\ 1\le i \le n-1, \label{eq:order-1} \allowdisplaybreaks\\
s_is_j&=s_js_i &\text{ for }\ j+1 < i \le n, \label{eq:cm-1} \allowdisplaybreaks\\
s_{ij}s_{i}& =s_{i-1}s_{ij} & \text{ for } j < i \le n-1 , \label{eq:s-1} \allowdisplaybreaks\\
s^{(k_1)}_{nj} \, s^{(k_2)}_{nj}& =s_{n-1}s^{(k_2)}_{nj}s^{(k_1)}_{n, j+1} & \text{ for } j\le n-1 \text{ and } k_1, k_2\ge 1. \label{eq:t-1}
\end{align}
\end{subequations}

\begin{prop} \label{ccan}
Using only the relations in $\mathcal R$, any element of the group $G(d,1,n)$ can be uniquely
written in the following reduced form
 \begin{equation} \label{eq:ccan}
   s_{1i_1}s_{2i_2}\cdots s_{n-1i_{n-1}}s^{(k_1)}_{nj_1}s^{(k_2)}_{nj_2}\cdots s^{(k_\ell)}_{n j_\ell}
 \end{equation}
 where $1 \le i_p \leq p+1$ for $1 \le p \le n-1$,  and $1\leq j_1<j_2<\cdots<j_\ell\leq n$ for $\ell\geq0$, and $1\le k_p \le d-1$ for $1 \le p \le \ell$.
\end{prop}

\begin{proof}
Consider $w \in G(d,1,n)$ and an expression of $w$ written in generators. Let $\ell$ be the number of occurrences of $s_n^k$ in the expression of $w$ for various $k$'s, where $k$ is maximal for each occurrence, i.e., if $w= \cdots s_i s_n^k s_j \cdots$, then $i \neq n$ and $j \neq n$. If $\ell =0$ then $w$ is an element of the subgroup of type $A_{n-1}$ and it is well known that one can use only the relations (without $s_n$) in $\mathcal R$  to obtain the reduced form \eqref{eq:ccan}. (See, for example, \cite{Bokut2001}.)

Assume that $\ell >0$.  Then we can write \[ w=w_1 s_n^k s_{1p_1}s_{2p_2}\cdots s_{n-1p_{n-1}}, \] where  $s_n^k$ is the last occurrence of a power of $s_n$ in the expression of $w$. By the commutativity relation \eqref{eq:cm-1}, we have
\[ 
w= w_1 s_{1p_1}s_{2p_2}\cdots s_{n-2p_{n-2}} s_n^k s_{n-1p_{n-1}}= w_2 \, s_{np_{n-1}}^{(k)},  
\]
where we set $w_2=w_1 s_{1p_1}s_{2p_2}\cdots s_{n-2p_{n-2}}$. By induction, the element $w_2$ can be written  in the form \eqref{eq:ccan}, and we have
\[ 
w=s_{1i_1}s_{2i_2}\cdots s_{n-1i_{n-1}}s^{(k_1)}_{nj_1}s^{(k_2)}_{nj_2}\cdots s^{(k_{\ell-1})}_{n j_{\ell-1}} s_{np_{n-1}}^{(k)} .
\] 
If $j_{\ell-1} < p_{n-1}$ or $j_{\ell-1}=n$ then we are done. If $j_{\ell-1} \ge p_{n-1}$ and $j_{\ell-1} < n$ then we use  the relations \eqref{eq:cm-1} and \eqref{eq:t-1} to obtain
\begin{align} 
s^{(k_{\ell-1})}_{n j_{\ell-1}} s_{np_{n-1}}^{(k)} & = s^{(k_{\ell-1})}_{n j_{\ell-1}} s_{nj_{\ell-1}}^{(k)} s_{j_{\ell-1}-1, p_{n-1}} \nonumber \\ 
&= s_{n-1} s_{nj_{\ell-1}}^{(k)} s^{(k_{\ell-1})}_{n, j_{\ell-1}+1} s_{j_{\ell-1}-1, p_{n-1}} = s_{n-1} s_{n p_{n-1}}^{(k)} s_{n, j_{\ell-1}+1}^{(k_{\ell-1})} . \label{eq:lo} 
\end{align} 
Using \eqref{eq:lo}, we  rewrite $w$ and apply the induction hypothesis again. We repeat this process until we get the canonical form \eqref{eq:ccan} in a finite number of steps.

Now we claim that the number of the canonical words is exactly $n!\cdot d^n$. Indeed, since each $i_p$ runs over the set $\{ 1,\cdots, p+1\}$, there are $2\cdot 3
\cdots n = n!$ choices for the part $s_{1i_1}s_{2i_2}\cdots s_{n-1i_{n-1}}$. Now for the part $s^{(k_1)}_{nj_1}s^{(k_2)}_{nj_2}\cdots s^{(k_\ell)}_{n j_\ell}$ with the conditions $1\leq j_1<j_2<\cdots<j_\ell\leq n$ ($\ell\geq0$) and $1 \le k_p \le d-1$,
we just need to consider the number of forms $s^{(k_1)}_{n1}s^{(k_2)}_{n2}\cdots s^{(k_\ell)}_{nn}$ with each $k_p$ running over the set $\{0,1,\cdots, d-1\}$,  setting $s^{(0)}_{nj_p}=1$ for convenience. In this way,
this part has $d^n$ elements. Thus altogether, we have $n!\cdot d^{n}$ canonical words as claimed.

We recall that $n! \cdot d^n$ is the order of $G(d,1,n)$. Thus we have shown that every element of $G(d,1,n)$ is uniquely written in the canonical form \eqref{eq:ccan}.
\end{proof}

\begin{remark} Proposition \ref{ccan} shows that $\mathcal R$ is a \GS basis for the group $G(d,1,n)$. See \cite{Bokut2001} for details about  \GS bases.
For the group $G(2,1,n)$, which is the Coxeter group of type $B_n$, the canonical form \eqref{eq:ccan} is obtained by Bokut and Shiao \cite[Lemma 5.2]{Bokut2001}. A different canonical form for the group $G(d,1,n)$ can be found in \cite{BM,KLLO,Sh}.

\end{remark}

\begin{dfn}
The set of canonical words in \eqref{eq:ccan} for $G(d,1,n)$ will be denoted by $\mathcal W(d,1,n)$.
The left factor $s_{1i_1}s_{2i_2}\cdots s_{n-1i_{n-1}}$  of a canonical word will be
called the \textit{prefix}, and similarly the right factor
$s^{(k_1)}_{nj_1}s^{(k_2)}_{nj_2} \cdots
s^{(k_\ell)}_{nj_\ell}$ will be called the \textit{suffix} of the
reduced word. Given a reduced word $\bw$ in the canonical form, we will
denote by $\pw$ the prefix of $\bw$ and by $\sw$ the suffix, and
write $\bw = \pw \sw$.
\end{dfn}

\subsection{Canonical reduced words for $G(d,d,n)$} \label{sec-g}

  For $1\leq i \leq n-1$, we define the words $\ts_{ij}$ in  the same way as with $s_{ij}$ using the generators $\ts_i$.
  When $i=n$, we define
\begin{equation} \label{eq:conv} \ts_{nj} =
 \left\{
   \begin{array}{ll}
     ~[n, n-2, \hdots, j] & \textrm{ if } j\leq n-2, \\
    ~[n]  & \textrm{ if } j=n,n-1, \\
     ~[~] & \textrm{ if } j>n.
   \end{array}
 \right.\end{equation}

  The group $G(d,d,n)$ can be embedded into $G(d,1,n)$ as  a subgroup of index $d$. Indeed, we define $\iota: G(d,d,n) \rightarrow G(d,1,n)$ by
\[ \iota (\ts_n) = s_n^{d-1} s_{n-1} s_n \quad \text{ and } \quad \iota (\ts_i)=s_i \text{ for }1 \le i \le n-1 . \] Then one can check that $\iota$ is a well-defined group homomorphism. Furthermore, we have the following lemma.

\begin{lemma}  \label{lem-tilde}
The homomorphism $\iota$ is an embedding and its image consists of the elements whose canonical words are in the set
\[\{    s_{1i_1}s_{2i_2}\cdots s_{n-1i_{n-1}}s^{(k_1)}_{nj_1}s^{(k_2)}_{nj_2}\cdots s^{(k_\ell)}_{n j_\ell} \in \mathcal W(d,1,n):  k_1+k_2+ \cdots +k_\ell \equiv 0 \, (\text{mod }d) \}.\] This set of reduced words will be denoted  by $\mathcal W(d,d,n)$.
\end{lemma}

\begin{proof}
Since $\iota (\ts_n)= s_{n,n-1}^{(d-1)} s_n$, it is clear from \eqref{eq:t-1} that the elements in the image of $\iota$ have reduced forms in $\mathcal W(d,d,n)$. We count the number of elements in $\mathcal W(d,d,n)$ and find that it is $n! \cdot d^{n-1}$ which is equal to the order of $G(d,d,n)$.  Thus $\iota$ is an embedding.
\end{proof}

The preimage of an element with a reduced word in $\mathcal W(d,d,n)$ can be found in the following way. We note that
\[ \iota (\ts_n \ts_{n-1})^k = s_{n-1} s_{n,n-1}^{(k)} s_n^{d-k} \]
and
\begin{align}
s^{(k_1)}_{nj_1}s^{(k_2)}_{nj_2} & =
s_{n-1}s_{n-1} s_{n,n-1}^{(k_1)} s_n^{d-k_1} s_n^{k_1} s_{n-2,j_1} s_{nj_2}^{(k_2)} \nonumber \\ &=
\iota(\ts_{n-1}) \, \iota(\ts_n\ts_{n-1})^{k_1} \, \iota(\ts_{n-2,j_1}) \, s_{nj_2}^{(k_1+k_2)} \nonumber \\ &= \iota(\ts_{n-1}\ts_{n})^{k_1} \, \iota(\ts_{n-1,j_1}) \,  s_{nj_2}^{(k_1+k_2)}. \label{eq:scm-1}
\end{align}
Using \eqref{eq:scm-1} repeatedly, we can write
\[ s_{1i_1}s_{2i_2}\cdots s_{n-1i_{n-1}}s^{(k_1)}_{nj_1}s^{(k_2)}_{nj_2}\cdots s^{(k_\ell)}_{n j_\ell} = \iota(\tilde w) s_{nj_\ell}^{(k_1+k_2+\cdots + k_\ell)} = \iota(\tilde w \ts_{n-1,j_\ell})  \]
for some $\tilde w \in G(d,d,n)$ with the condition $k_1+k_2+ \cdots +k_\ell \equiv 0 \, (\text{mod }d)$.

From now on, the group $G(d,d,n)$ will be identified with the image of $\iota$ and the set $\mathcal W(d,d,n)$ will be the set of {\em canonical words} for $G(d,d,n)$. As in the case of $G(d,1,n)$, we write $\bw=\pw \sw \in \mathcal W(d,d,n)$ as a product of the prefix $\pw$ and the suffix $\sw$.

\begin{remark}
One can try to obtain canonical words for $G(d,d,n)$ without using an embedding into $G(d,1,n)$. However, we find that a natural set of {\em reduced} words thus obtained is not compatible with the packet decomposition defined in Section \ref{sec-gddn}. See Appendix for $G(3,3,3)$ as an example.

\end{remark}

\subsection{Canonical reduced words for $G(d,r,n)$} \label{sec-gr}

  The group $G(d,r,n)$ can be embedded into $G(d,1,n)$ as  a subgroup of index $r$. Indeed, we define $\tau: G(d,r,n) \rightarrow G(d,1,n)$ by
\[ \tau (\ts) = s_n^r, \quad \tau (\ts_n) = s_n^{d-1} s_{n-1} s_n \quad \text{ and } \quad \tau (\ts_i)=s_i \text{ for }1 \le i \le n-1 . \] Then one can check that $\tau$ is a well-defined group homomorphism.

\begin{lemma}  \label{lem-tilde-r}
The homomorphism $\tau$ is an embedding and its image consists of the elements whose canonical words are in the set
\[ \mathcal W(d,r,n) :=\{    s_{1i_1}s_{2i_2}\cdots s_{n-1i_{n-1}}s^{(k_1)}_{nj_1}s^{(k_2)}_{nj_2}\cdots s^{(k_\ell)}_{n j_\ell} \in \mathcal W(d,1,n):  k_1+k_2+ \cdots +k_\ell \equiv 0 \, (\text{mod }r) \}. \]
\end{lemma}

\begin{proof}
Since $\tau (\ts_n)= s_{n,n-1}^{(d-1)} s_n$ and $\tau (\ts)= s_n^{r}$, it is clear that the image of $\tau$ is contained in $\mathcal W(d,r,n)$. One sees that the number of elements in $\mathcal W(d,r,n)$ is $n! \cdot d^{n}/r$ which is equal to the order of $G(d,r,n)$.  Thus $\tau$ is an embedding.
\end{proof}

As with the map $\iota$ for $G(d,d,n)$,  we have
\begin{align}
s^{(k_1)}_{nj_1}s^{(k_2)}_{nj_2} & = \tau(\ts_{n-1}\ts_{n})^{k_1} \, \tau(\ts_{n-1,j_1}) \, s_{nj_2}^{(k_1+k_2)}. \label{eq:scm-1-r}
\end{align}
Using \eqref{eq:scm-1-r} repeatedly, we can write
\[ s_{1i_1}s_{2i_2}\cdots s_{n-1i_{n-1}}s^{(k_1)}_{nj_1}s^{(k_2)}_{nj_2}\cdots s^{(k_\ell)}_{n j_\ell} = \tau(\tilde w) s_{nj_\ell}^{(k_1+k_2+\cdots + k_\ell)} = \tau(\tilde w \, \ts^k \, \ts_{n-1,j_\ell})  \]
for some $\tilde w \in G( d,r ,n)$ and $k \in \mathbb Z_{\ge 0}$ with the condition $k_1+k_2+ \cdots +k_\ell \equiv 0 \, (\text{mod }r)$.

From now on, the group $G(d,r,n)$ will be identified with the image of $\tau$ and an element of the set $\mathcal W(d,r,n)$ will be called a {\em canonical word}. As in the case of $G(d,1,n)$, a canonical word $\bw$ will be written as $\bw = \pw \sw$, where  $\pw$ is the prefix and $\sw$ the suffix of $\bw$.

\section{Fully commutative elements} \label{sec-fc}

\subsection{Case $G(d,1,n)$} In this subsection, let $W \coloneqq G(d,1,n)$ be the complex reflection group defined in the previous section,
with  $S\coloneqq \{ s_1,\ldots,s_{n-1},s_n\} $ the set of generators and the defining relations in \eqref{eq:order1}, \eqref{eq:cm1}, \eqref{eq:b1} and \eqref{eq:l1}.
We consider the free monoid $S^\star$ consisting of all finite length words
$w=s_{i_1}s_{i_2}\cdots s_{i_\ell}$ with $s_{i_j} \in S$. The multiplication in $S^\star$ is defined by the concatenation \[(s_{i_1}\cdots s_{i_\ell})\cdot(s_{m_1}\cdots s_{m_t})= s_{i_1}\cdots s_{i_\ell}s_{m_1}\cdots s_{m_t}.\]

We define a binary relation $\approx$ on $S^\star$ generated by
the relations
\begin{subequations}
\begin{align}
s_is_j&=s_js_i \quad &\mbox{for}\ j+1 < i \le n, \label{eq:cm-2} \allowdisplaybreaks \\
s_{i+1}s_is_{i+1}&=s_is_{i+1}s_i \quad &\mbox{for}\ 1\le i \le n-2, \label{eq:b-2} \allowdisplaybreaks \\
s^{k_1}_ns_{n-1}s^{k_2}_ns_{n-1}&=s_{n-1}s^{k_2}_ns_{n-1}s^{k_1}_n \quad &\mbox{for}\ 1\le k_1, k_2 <d. \label{eq:l-2}
\end{align}
\end{subequations}
Define $R(w) \subset S^\star$ to be the set of reduced
expressions for $w \in W$. Here, as usual, a reduced expression is a word of minimal length for $w$.

\medskip

The following proposition is a generalization of Matsumoto's Theorem and is
crucial to define fully commutative elements.

\begin{prop} \label{prop-rw}
For any $w \in W$, the set $R(w)$ has exactly one equivalence class under $\approx$.
\end{prop}

\begin{proof}
Suppose that $\bw \in R(w)$. We will show that $\bw$ is related to the canonical word in $\mathcal W(d,1,n)$ under $\approx$. We follow the proof of Proposition \ref{ccan}.
Let $\ell$ be the number of occurrences of $s_n^k$ in $\bw$ for various $k$'s, where $k$ is maximal for each occurrence.
If $\ell=0$ then $\bw$ is a reduced word of  an element of the subgroup of type $A_{n-1}$ and one can use only the relations \eqref{eq:cm-2} and \eqref{eq:b-2} to obtain the canonical word by  Matsumoto's Theorem.

Assume $\ell >0$. Then the proof of Proposition \ref{ccan} shows that the relations \eqref{eq:cm-1} and \eqref{eq:t-1}  are used to obtain the canonical word. As the relation \eqref{eq:cm-1} is nothing but \eqref{eq:cm-2}, we have only to check if the relation \eqref{eq:t-1} is derived from \eqref{eq:cm-2}, \eqref{eq:b-2} and \eqref{eq:l-2}. That is already  done in Lemma \ref{lem-only}.
\end{proof}

\begin{ex} \label{ex-over}
Consider $W=G(3,1,2) = <s_1, s_2 \,|\, s_2^3=s_1^2=1,\ s_2s_1s_2s_1=s_1s_2s_1s_2 >$. Then the two reduced expressions $s_2s_1s_2^2s_1$ and $s_1s_2^2s_1s_2$ represent
the same element in $W$. One cannot be transformed into the other, using  only the defining braid relations \eqref{eq:cm1}, \eqref{eq:b1} and \eqref{eq:l1}. However, under $\approx$, the two expressions are related through \eqref{eq:l-2}.
\end{ex}

We define a weaker binary relation $\sim$ on $S^\star$ generated by the
relations \eqref{eq:cm-2} only. The equivalence classes under this relation
are called {\it commutativity classes}. This gives the decomposition of $R(w)$
into commutativity classes:
$$ R(w) = \mathcal C_1 \dot\cup \  \mathcal C_2 \cdots \dot\cup \ \mathcal C_\ell.$$

\begin{dfn}

We say that  $w \in W$ is {\em fully commutative} if $R(w)$ consists of a single commutativity
class; i.e., any reduced word for $w$ can be obtained from any other solely by use of
the braid relations that correspond to commuting generators.

\end{dfn}

Throughout this paper, a subword always means a subword  with all its letters in consecutive positions.
We obtain  the following lemma which is an analogue of Proposition 2.1 in \cite{Stem96}.

\begin{lemma} \label{lem-fc}
An element $w\in W$ is fully commutative if and only if no member of $R(w)$ contains $s_{i+1}s_is_{i+1}$, $1 \le i \le n-2$, or $s_n^{k_1}s_{n-1}s_n^{k_2}s_{n-1}$, $1 \le k_1, k_2 <d$,
as a subword.
\end{lemma}

\begin{proof} We will prove the contrapositive.
If a word $\bw \in R(w)$ has such a subword, the word $\bw$ cannot be transformed into the canonical form only using commutative relations. Thus
$w$ cannot be fully commutative.
Conversely, if $w$ is not fully commutative, there must be a word $\bw \in R(w)$ to which one of the relations in \eqref{eq:b-2} and \eqref{eq:l-2} is applied. Then we see that there exists a word $\bw_1$ obtained from $\bw$, which has $s_{i+1}s_is_{i+1}$, $1 \le i \le n-2$, or $s_n^{k_1}s_{n-1}s_n^{k_2}s_{n-1}$, $1 \le k_1, k_2 <d$,
as a subword.
\end{proof}

The following proposition provides a practical criterion for full commutativity.

\begin{prop} \label{pro-strong}
An element $w\in W$ is fully commutative if and only if there exists $\bw \in R(w)$ whose commutativity class has no member that has as a subword any of the following
\begin{subequations}
\begin{align}
&s_{i+1}s_is_{i+1}, & &s_{i}s_{i+1}s_{i}& (1 \le i \le n-2), \label{rlis} \allowdisplaybreaks\\ 
&s_n^{k_1}s_{n-1}s_n^{k_2}s_{n-1},& & s_{n-1}s_n^{k_1}s_{n-1}s_n^{k_2} & (1 \le k_1, k_2 <d). \label{rlis-1}
\end{align}
\end{subequations}
\end{prop}

\begin{proof}
Assume that $w$ is fully commutative. Consider $\bw \in R(w)$. Then by Lemma \ref{lem-fc}, the word $\bw$ does not contain $s_{i+1}s_is_{i+1}$, $1 \le i \le n-2$, or $s_n^{k_1}s_{n-1}s_n^{k_2}s_{n-1}$, $1 \le k_1, k_2 <d$,
as a subword. Moreover, $\bw$ cannot contain $s_is_{i+1}s_i$, $1 \le i \le n-2$, or $s_{n-1}s_n^{k_1}s_{n-1}s_n^{k_2}$, $1 \le k_1, k_2 <d$, either. If it does, we obtain $\bw_1 \in R(w)$ from $\bw$ by applying \eqref{eq:b-2} or \eqref{eq:l-2}, which contains $s_{i+1}s_is_{i+1}$, $1 \le i \le n-2$, or $s_n^{k_1}s_{n-1}s_n^{k_2}s_{n-1}$, $1 \le k_1, k_2 <d$. That is a contradiction. Thus any $\bw \in R(w)$ does not contain any of the words in \eqref{rlis} and \eqref{rlis-1}.

Conversely, assume that there exists  $\bw \in R(w)$ whose commutativity class $\mathcal C$ has no member that contains any of the words in \eqref{rlis} and \eqref{rlis-1}. Then neither of \eqref{eq:b-2} or \eqref{eq:l-2} can be applied to any of the member of $\mathcal C$, and we must have $\mathcal C =R(w)$. Thus, by definition, $w$ is fully commutative.
\end{proof}

As in Section \ref{first}, an expression $s_{i_1} \cdots
s_{i_r} \in W$ will be identified with the word $[i_1,  \dots, i_r]$.
For $w \in W$, let $\bw =[i_1, \dots, i_r] \in R(w)$. Define $\{i,i+1 \}$-sequence of $\bw$ to be the sequence of $i$'s and $i+1$'s obtained by ignoring all entries of $\bw$ different from $i$ and $i+1$. For example, the $\{1,2\}$-sequence of $\bw=[1,2,1,3,4,3,2]$ is $[1,2,1,2]$.

We have the following useful lemma due to Kleshchev and Ram.

\begin{lemma}[\cite{KR}] \label{lem-kr}
The reduced words $\bw$ and $\bv$ are in the same commutativity class if and only if their $\{i, i+1\}$-sequences coincide for each $i=1,2, \dots, n-1$.
\end{lemma}

Combining Proposition \ref{pro-strong} and Lemma \ref{lem-kr}, we can easily check whether an element $w$ is fully commutative or not.

In Proposition \ref{ccan}, we prove that any element of $W$ can be written as
 \begin{equation} \label{eqn-canonical}
   s_{1i_1}s_{2i_2}\cdots s_{n-1i_{n-1}}s^{(k_1)}_{nj_1}s^{(k_2)}_{nj_2}\cdots s^{(k_\ell)}_{n j_\ell}
 \end{equation}
 where $1 \le i_p \leq p+1$ for $1 \le p \le n-1$,  and $1\leq j_1<j_2<\cdots<j_\ell\leq n$ for $\ell\geq0$, and $1 \le k_p \le d-1$ for $1 \le p \le \ell$. Here we write \[ s^{(k)}_{nj} =s_n^k s_{n-1,j} =s_n^k s_{n-1}\ldots s_{j}\qquad \text{ for } k \ge 1\,\,\text{ and}\,\,\, j \le n. \]
Recall that the part $ s^{(k_1)}_{nj_1}s^{(k_2)}_{nj_2}\cdots s^{(k_\ell)}_{n j_\ell}$ in the canonical form \eqref{eqn-canonical} is called its suffix. Now  we prove the following proposition.

\begin{prop} \label{every suffix}
  Every  suffix is a fully commutative element.
\end{prop}

\begin{proof}
By Proposition \ref{pro-strong}, we have only to show that no member in the commutative class of a suffix contains as a subword any of the words in \eqref{rlis} and \eqref{rlis-1}. By Lemma \ref{lem-kr}, we need only to investigate relative positions of a letter $p$ and their neighbors $p-1$ and $p+1$.

Assume that $1 \le p \le n-2$. Every consecutive occurrence of $p$ in a suffix  is of the form
\begin{equation} \label{repeat} [p, p-1, \ldots j_p, n^{k_p}, \ldots,  p+1, p, \ldots, j_{p+1}] \quad \text{ with } \ j_p <  j_{p+1} < p+1.\end{equation}  Thus the word $s_p s_{p+1}s_p$ or $s_{p+1} s_p s_{p+1}$ cannot appear in  any suffix or in any member of its commutative class.
Similarly, if $p=n-1$, then the word $s_{n-1}s_{n-2}s_{n-1}$ or $s_{n-2}s_{n-1}s_{n-2}$ cannot appear, as the only difference is that $p+1=n$ may be repeated in \eqref{repeat}.
Finally, one sees that the words $s_{n-1}s_n^{k_1}s_{n-1}s_n^{k_2}$ or $s_n^{k_1}s_{n-1}s_n^{k_2}s_{n-1}$
cannot appear in  any suffix or in its commutative class.
Hence our assertion follows.

\end{proof}

\subsection{Cases $G(d,d,n)$ and $G(d,r,n)$}

Recall that we fixed embeddings of $G(d,d,n)$ and $G(d,r,n)$ into $G(d,1,n)$ and that these groups are identified with the images of the embeddings.

\begin{dfn}
Let $W= G(d,d,n)$ or $G(d,r,n)$ be considered as a subgroup of $G(d,1,n)$ through the embedding $\iota$ or $\tau$ defined in Section \ref{first}, respectively. An element $w$ of $W$ is called {\em fully commutative} if $w$ is fully commutative as an element of $G(d,1,n)$.
\end{dfn}

As mentioned in the introduction, this definition of fully commutative elements coincides with the usual definition for the Coxeter groups of type $B_n$ when $d=2$, $r=1$. On the other hand, it is not compatible with  the usual definition  for the Coxeter groups of type $D_n$ when $d=2$, $r=2$. This will be made more clear in the following sections.

\section{Packets in  $G(d,1,n)$}\label{ch:D}

\subsection{Collections}
The words in $\mathcal {W}(d,1,n)$ which correspond to fully commutative elements will be called {\it fully commutative} and will be grouped based on
their suffixes.

\begin{dfn}
A {\em collection} $\coll{n}{\sw} \subset \mathcal {W}(d,1,n)$ labeled by
a suffix $\sw$ is defined to  be the set of fully commutative words in
$\mathcal {W}(d,1,n)$ whose  suffix is $\sw$.
\end{dfn}

As in the case of type $D$ studied in \cite{FL14}, some of the collections have the same number of elements as we will
prove in the rest of this subsection. The proofs are essentially the same as in the case of type $D$, and we will only sketch the proofs, referring the reader to \cite{FL14} for more detailed proofs.

\begin{lemma}\label{type2}
 For a fixed $k$, $0\leq k \leq n-2$, any collection labeled by a suffix of the form
\begin{equation} \label{ff}
s^{(t_1)}_{n k+1}s^{(t_2)}_{nj_2}s^{(t_3)}_{nj_3}s^{(t_4)}_{nj_4}\cdots s^{(t_\ell)}_{nj_\ell}\quad (\ell \ge 2)
\end{equation}
has the same set of prefixes. In particular, these collections have the same number of elements.
\end{lemma}

\begin{proof}

Let $\bw'$ be a suffix of the form \eqref{ff}. Then $\bw'$ has the
suffix $\bw_1:=s^{(t_1)}_{n\, k+1} s_{n}$ as a subword, and any prefix appearing in the collection
$\coll{n}{\bw'}$ also appears in $\coll{n}{\bw_1}$.

Conversely, assume that $\pw$ is a prefix of a fully commutative word appearing in the collection $\coll{n}{\bw_1}$.
Since the prefix and suffix of a fully commutative word are individually fully commutative, we only assume
that there is some letter $r \le n-1$ which appears in both $\pw$ and $\sw$.  From the condition $$1\leq j_1<j_2<\cdots<j_\ell\leq n$$
on the suffix $\sw$, we see that the letter $r$ also appears in $\bw_1$.
Then the  full commutativity of $\bw_0\bw_1$ implies that of $\pw \sw$, and $\pw$ is a prefix of $\sw$.
\end{proof}

\begin{prop}\label{type size}
For $1 \leq k\le n-2$, the collection labeled by the suffix $s^{(t')}_{nk}$, $1 \le t' <d$, has the same number of elements as any of the collections labeled by the suffix of the form
$$s^{(t_1)}_{n k+1}s^{(t_2)}_{nj_2}s^{(t_3)}_{nj_3}s^{(t_4)}_{nj_4}\cdots s^{(t_\ell)}_{nj_\ell}\quad (\ell \ge 2). $$
\end{prop}

\begin{proof}
Let $\bw_1= s^{(t_1)}_{n \, k+1}s_{n}$ and $\bw_2=s^{(t')}_{nk}$.  By Lemma
\ref{type2}, it is enough to establish a bijection between the
collections $\coll{n}{\bw_1}$ and $\coll{n}{\bw_2}$. We define a map
$\sigma: \coll{n}{\bw_2} \to \coll{n}{\bw_1}$ as follows.
  Suppose that $ \pw$ is the prefix of the word $\bw=\pw \bw_2=\pw [n^{t'}, n-1, \dots , k] \in \coll{n}{\bw_2}$,
  and let $r$ be the last letter of $\pw$. Then by the condition of full commutativity, we must have $r<k$ or $r=n-1$.

If $r <k$ we simply define $\sigma(\bw)=\pw \bw_1$. If $r=n-1$, we take $m\geq k$ to be the smallest letter such that the string $[m, m+1, \hdots, n-1]$ is a right factor of $\pw$.
Then we have $\bw=s_{1i_1}\cdots s_{m-1i_{m-1}} s_ms_{m+1} \cdots s_{n-1} \bw_2$,  and we define
$$  \sigma(\bw) = s_{1i_1}\cdots s_{m-1i_{m-1}}s_{mk}\bw_1. $$
Then we have $\sigma(\coll{n}{\bw_2})\subset \coll{n}{\bw_1}$.

 Now we define a map $\eta:\coll{n}{\bw_1} \to  \coll{n}{\bw_2}$.
  Suppose that $\bw = \pw\bw_1=\pw[n^{t_1}, n-1, \hdots, k+1, n] \in \coll{n}{\bw_1}$,
and let $r$ be the last letter of $\pw$. Then by the condition of
full commutativity, we must have $1\le r \le k$.

 If $r<k$, then we define $\eta(\pw \bw_1)=\pw \bw_2$.
If $r=k$, then the final non-empty segment of the prefix is $s_{mk}$ for some $m$ with $k \leq m \leq n-1$.  We define
\[\eta(\bw) =\eta( s_{1i_1}\cdots s_{m-1i_{m-1}}s_{mk}\bw_1)= s_{1i_1}\cdots s_{m-1i_{m-1}}s_{m}s_{m+1} \cdots s_{n-1} \bw_2. \]
One sees that $\eta(\coll{n}{\bw_1}) \subset  \coll{n}{\bw_2}$, and that $\eta$ is both a left and a right inverse of
$\sigma$.
\end{proof}

\begin{lemma}\label{type size2}
The collections labeled by the suffixes $s^{t}_{n}$  and $s^{(t)}_{n, n-1}$ $(1 \le t <d)$ all have the same set of prefixes.
\end{lemma}

\begin{proof}
Assume that $\pw$ is a prefix of $s^{t}_n$ or $s^{(t)}_{n, n-1}$, i.e., $\pw s^{t}_n$  or $\pw s^{(t)}_{n, n-1}$ is fully commutative. Then replacing the suffix with any of the suffixes $s^{t'}_{n}$  and $s^{(t')}_{n, n-1}$ $(1 \le t' <d)$ does not affect full commutativity.
\end{proof}

\subsection{Packets} \label{sec-pac}

The results in the previous subsection lead us to the following definition.

\begin{dfn} \label{dfn-lp}
For $0 \le k \le n$, we define the {\em $(n,k)$-packet} of
collections:
\begin{itemize}
\item The $(n,0)$-packet is the set of collections labeled by suffixes of the form
\[s^{(t_1)}_{n1}s^{(t_2)}_{nj_2}s^{(t_3)}_{nj_3}s^{(t_4)}_{nj_4}\cdots s^{(t_\ell)}_{nj_\ell}\quad (\ell \ge 2).\]
\item The $(n,k)$-packet, $1\le k \le n-2$, is the set of collections labeled by $s^{(t)}_{nk}$ or suffixes of the form
$s^{(t_1)}_{n\, k+1}s^{(t_2)}_{nj_2}s^{(t_3)}_{nj_3}s^{(t_4)}_{nj_4}\cdots s^{(t_\ell)}_{nj_\ell}\quad (\ell \ge 2)$.
\item The $(n,n-1)$-packet contains the collections labeled by $s^{(t)}_n=[n^t]$ or $s^{(t)}_{n,n-1}=[n^t,n-1]$.
\item The $(n,n)$-packet contains only the collection labeled by the empty suffix $[~]$.
\end{itemize}
We will denote the $(n,k)$-packet by $\PP(n,k)$.  As an example,
Table~\ref{G213-packets} shows all of the packets in the case of $G(2,1,3)$ (or $B_3$).
\end{dfn}

\begin{table}
\pgfdeclarelayer{background}
\pgfdeclarelayer{foreground}
\pgfsetlayers{background,main,foreground}
\tikzstyle{coll} = [draw, very thick, fill=white,
    rectangle, rounded corners, inner sep=18pt, inner ysep=15pt, minimum height=2in]
\tikzstyle{fancytitle} =[fill=black, text=white]

\[
\begin{tikzpicture}[font=\footnotesize]

\node[coll,minimum height=.6cm, inner sep = 2mm] (C4213){
\begin{minipage}{.2\textwidth}\begin{center}
\underline{$\mathbf c^{3}_{[3,2,1,3]}$}\\[2mm]
 $[3,2,1,3]$
 \end{center}
 \end{minipage}
};

\node[coll, minimum height=.6cm, right of=C4213, xshift=3cm, inner sep = 2mm] (C42132){
\begin{minipage}{.2\textwidth}\begin{center}
\underline{$\mathbf c^{3}_{[3,2,1,3,2]}$}\\[2mm]
 $[3,2,1,3,2]$
 \end{center}
 \end{minipage}
};

\node[coll, minimum height=.6cm, right of=C42132, xshift=3cm, inner sep = 2mm] (C421324){
\begin{minipage}{.2\textwidth}\begin{center}
\underline{$\mathbf c^{3}_{[3,2,1,3,2,3]}$}\\[2mm]
 $[3,2,1,3,2,3]$
 \end{center}
 \end{minipage}
};

\node[coll, below of=C4213, yshift=-1.8cm, xshift=.85cm, minimum height=1cm, inner sep = 2mm] (C421){
\begin{minipage}{.35\textwidth}\begin{center}
\underline{$\mathbf c^{3}_{[3,2,1]}$}\\[2mm]
\begin{multicols}{3}
   $[3,2,1]$\\[2pt]
   $[2,3,2,1]$\\[2pt] \columnbreak
   $[1,2,3,2,1]$\\[2pt]
   \end{multicols}
 \end{center}
 \end{minipage}
};

\node[coll, right of=C421, xshift=5.5cm, minimum height=1cm, inner sep=2mm] (C423){
\begin{minipage}{.35\textwidth}\begin{center}
\underline{$\mathbf c^{3}_{[3,2,3]}$}\\[2mm]
\begin{multicols}{3}
   $[3,2,3]$\\[2pt]
   $[1,3,2,3]$\\[2pt]
   $[2,1,3,2,3]$\\[2pt]
\end{multicols}
 \end{center}
 \end{minipage}
};

\node[coll, below of=C42132, yshift=-4.8cm,
minimum height=1cm, inner sep=3mm] (C42){
\begin{minipage}{.65\textwidth}\begin{center}
\underline{$\mathbf c^{3}_{[3,2]}$}\\[2mm]
\begin{multicols}{5}
   $[3,2]$\\[2pt]
   $[2,3,2]$\\[2pt]
   $[1,3,2]$\\[2pt]
   $[2,1,3,2]$\\[2pt]
   $[1,2,3,2]$\\[2pt]
   \end{multicols}
 \end{center}
 \end{minipage}
};

\node[coll, below of=C42, yshift=-1cm,
minimum height=1cm, inner sep=3mm] (C4){
\begin{minipage}{.65\textwidth}\begin{center}
\underline{$\mathbf c^{3}_{[3]}$}\\[2mm]
\begin{multicols}{5}
   $[3]$\\[2pt]
   $[2,3]$\\[2pt]
   $[1,3]$\\[2pt]
   $[2,1,3]$\\[2pt]
   $[1,2,3]$\\[2pt]
   \end{multicols}
 \end{center}
 \end{minipage}
};

\node[coll, below of=C4, yshift=-2cm,
minimum height=1cm, inner sep=3mm] (C){
\begin{minipage}{.65\textwidth}\begin{center}
\underline{$\mathbf c^{3}_{[~]}$}\\[2mm]
\begin{multicols}{5}
   $[~]$\\[2pt]
   $[2]$\\[2pt]
   $[1]$\\[2pt]
   $[2,1]$\\[2pt]
   $[1,2]$\\[2pt]
   \end{multicols}
 \end{center}
 \end{minipage}
};

\begin{pgfonlayer}{background}

\path (C4213.north -| C.west)+(-.5cm,.4cm) node (ul1) {};
\path (C421324.south -| C.east)+(.5cm,-.3cm) node (lr1) {};
\draw[very thick, rounded corners, fill=black!20] (ul1) rectangle (lr1);
\node[fancytitle, xshift=.5cm,  anchor=west] at (ul1) (P1title) {$\mathcal P(3,0)$};

\path (C421.north -| C.west)+(-1cm,.4cm)node (ul2) {};
\path (C423.south -| C.east)+(1.2cm,-.3cm) node (lr2) {};
\draw[very thick, rounded corners, fill=black!20] (ul2) rectangle (lr2);
\node[fancytitle, xshift=.5cm,  anchor=west] at (ul2) (P2title) {$\mathcal P(3,1)$};

\path (C42.north -| C.west)+(-.5cm,.4cm)node (ul3) {};
\path (C4.south -| C.east)+(.5cm,-.3cm) node (lr3) {};
\draw[very thick, rounded corners, fill=black!20] (ul3) rectangle (lr3);
\node[fancytitle, xshift=.5cm,  anchor=west] at (ul3) (P3title) {$\mathcal P(3,2)$};


\path (C.north west)+(-.5cm,.4cm)node (ul5) {};
\path (C.south east)+(.5cm,-.3cm) node (lr5) {};
\draw[very thick, rounded corners, fill=black!20] (ul5) rectangle (lr5);
\node[fancytitle, xshift=.5cm,  anchor=west] at (ul5) (P5title) {$\mathcal P(3,3)$};

\end{pgfonlayer}
\end{tikzpicture}
\]
\caption{The packets of $G(2,1,3)$}\label{G213-packets}
\end{table}

We record the main property of  a packet as a corollary.
\begin{cor} \label{num}
The collections in a fixed packet $\PP(n,k)$ have the same number of elements.
\end{cor}

\begin{proof}
The assertion follows from Lemma \ref{type2}, Proposition \ref{type size}  and Lemma \ref{type size2}.
\end{proof}

We count the number of collections in a packet and obtain:
\begin{prop} \label{prop-d-1}
The size of the packet $\PP(n,k)$ of $G(d,1,n)$ is
$$ \left | \PP(n,k) \right| =
   \left\{\begin{array}{cl}
   (d^{n-1}-1)(d-1) & \textrm{ if } k = 0, \\
   d^{n-k-1}(d-1)  & \textrm{ if } 1\le k \le n-2, \\
   2(d-1) & \textrm{ if } k= n-1, \\
   1 & \textrm{ if } k= n. \\
   \end{array}
   \right.
$$
Hence we have $ \sum_{k=0}^n  \left | \PP(n,k) \right|  =  d^{n}$.
\end{prop}

\begin{proof}
Assume that $k=0$. We consider the expression
\[s^{(k_1)}_{n1}s^{(k_2)}_{n2}s^{(k_3)}_{n3}\cdots  s^{(k_{n-1})}_{n,n-1} s^{k_n}_{n}.\]
The conditions for $(n,0)$-packet allows $k_1$ to vary from $1$ to $d-1$ and $k_i$ $(2 \le i \le n)$ from $0$ to $d$ except the case $k_2=k_3=\cdots =k_n=0$.
Thus the total number of collections in $\PP(n,0)$ is $(d-1) (d^{n-1}-1)$.

Similar arguments can be applied to the other packets $\PP(n,k)$ for $1 \le k \le n-1$, and it is clear that there is only one collection in $\PP(n,n)$. The total sum can be checked straightforwardly.
\end{proof}

\subsection{Catalan's Triangle}
In this subsection, we will compute the size of a collection in a
given packet, and thereby  classify and enumerate all the fully commutative elements.

As in the case of type $D$ studied in \cite{FL14}, the sizes of collections are given by {\it Catalan triangle numbers} $C(n,k)$ which are defined by
\begin{equation} \label{cat form}
  C(n,k) = \frac{(n+k)!(n-k+1)}{k!(n+1)!}
\end{equation}
for $n \geq 0$ and $0 \leq k \leq n$.
The numbers form the {\em Catalan Triangle} in Table \ref{CT}  to satisfy the rule:
\begin{equation}\label{sum rule}
    C(n,k) = C(n,k-1)+C(n-1,k),
\end{equation}
where all entries outside of the range $0\leq k \leq n$ are considered to be $0$. One also sees that
\begin{equation} \label{ccc}
C_n=C(n, n-1)=C(n,n) ,
\end{equation}
where $C_n$ is the $n^{\textrm{th}}$ Catalan number.

\begin{table}[h]
\[\begin{array}{ccccccccc}
1\\
1 & 1\\
1 & 2 & 2\\
1 & 3 & 5 & 5\\
1 & 4 & 9 & 14 & 14\\
1 & 5 & 14 & 28 & 42 & 42\\
1 & 6 & 20 & 48 & 90 & 132 & 132\\
1 & 7 & 27 & 75 & 165 & 297 & 429 & 429\\
\vdots & \vdots &\vdots &\vdots &\vdots &\vdots &\vdots & \vdots& \ddots\\

\end{array}
\]
\caption{Catalan Triangle} \label{CT}
\end{table}

\begin{lemma} \label{diag}
  For $n\geq 3$, the size of each collection in the packets $\PP(n,n)$ and $\PP(n,n-1)$  is equal to  the Catalan number $C_n$.
\end{lemma}

\begin{proof}
The proof of Lemma \ref{type size2} shows that prefixes in a collection belonging to one of the packets $\PP(n,n)$ and $\PP(n,n-1)$  are exactly fully commutative words of type $A_{n-1}$, the total number of which is well known to be the Catalan number $C_n$.
\end{proof}

The following theorem is an extension of Theorem 2.12 in \cite{FL14} from the case of $D_n$ to $G(d,1,n)$.  The proof is similar to that of type $D_n$, and we refer the reader to \cite{FL14} for more details.

\begin{thm} \label{main D} Assume that $n\ge 3$ and $0\le k \le n$. Then any collection in the packet $\PP(n,k)$ contains exactly $C(n,k)$
  elements.
\end{thm}

\begin{proof}
We have already proved the cases when $k=n$ and $n-1$ in Lemma~\ref{diag}.
Now consider the packet $\PP(n,0)$, which consists of the collections labeled by the suffixes
$$s^{(t_1)}_{n1}s^{(t_2)}_{nj_2}s^{(t_3)}_{nj_3}s^{(t_4)}_{nj_4}\cdots s^{(t_\ell)}_{nj_\ell} \  \ (\ell \ge 2).$$
By Lemma \ref{type2}, it is enough to consider the collection $\bf{c}$ labeled by $s^{(t_1)}_{n1}s_{n}$.
If a word $\bw \in \bf{c}$ contains a non-empty prefix
$\pw$ ending with the letter $r$ for $1\leq r \leq n-1$, then the
word $\bw$ contains, as a right factor, $[ r, n, n-1,\dots, r+1, r,  \dots, 1,
n]$  which contradicts full commutativity. Thus
the collection $\bf{c}$, and hence every collection in $\PP(n,0)$,
contains only the suffix itself. Thus we have $C(n,0)=1$.

For the other cases, we will combinatorially (or bijectively) realize the identity \eqref{sum rule}.
One can see that any collection in the packet
$\PP(3,k)$ contains exactly $C(3,k)$ elements for $0 \le k \le 3$. The case $G(2,1,3)$ is given in Table \ref{G213-packets} and the case $G(3,1,3)$ in the second example of Section \ref{sec:ex}.  Thus the assertion of
the theorem is true for $n=3$, and we will proceed by induction with  the base cases $k=0$ or $n=3$.  Further, by Corollary \ref{num}, it is  enough to consider a single collection
in each of the packets.

Assume that $n \ge 4$  and $1 \leq k \leq n-2$.  We define
\begin{align*} & \bw_1= s_{nk}s_{n}=[n, n-1, \dots, k, n],  \quad \bw_2=s_{nk}=[n, n-1, \dots , k], \\
&  \bw_3= [n-1, n-2, \dots , k] .
\end{align*}
Then $\coll{n}{\bw_1} \in \PP(n, k-1)$, $\coll{n}{\bw_2} \in \PP(n, k)$ and $\coll{n-1}{\bw_3} \in \PP(n-1, k)$.
We will give an explicit bijection from  $\coll{n}{\bw_1}\cup\coll{n-1}{\bw_3}$ to $\coll{n}{\bw_2}$.

Define a map $\varphi_1: \coll{n}{\bw_1} \to \coll{n}{\bw_2}$ by
$$
    \varphi_1(\pw \bw_1) = \varphi(\pw [n, n-1, \dots , k, n] )= \pw \bw_2 = \pw [n, n-1, \dots , k] ,
$$
and another map $\varphi_2: \coll{n-1}{\bw_3} \to  \coll{n}{\bw_2}$ by
$$\varphi_2(\pw\bw_3) = \pw s_{n-1} \bw_2= \pw s_{n-1} [ n, n-1, n-2, \dots, k].$$
Then it can be checked that the maps $\varphi_1$ and $\varphi_2$ are well defined ant that the images are disjoint.
Finally, combining $\varphi_1$ and $\varphi_2$, we define a map $\varphi:
\coll{n}{\bw_1}\cup\coll{n-1}{\bw_3} \to \coll{n}{\bw_2}$, i.e. the restriction of
$\varphi$  to $\coll{n}{\bw_1}$ is $\varphi_1$ and the
restriction of $\varphi$ to $\coll{n-1}{\bw_3}$ is
$\varphi_2$.

Conversely, define the  map $\rho: \coll{n}{\bw_2} \to \coll{n}{\bw_1}\cup\coll{n-1}{\bw_3}$ to be given by the rule:
$$
   \rho(\pw\bw_2) = \rho(\pw [n, n-1, \dots , k]) =
   \left\{\begin{array}{ll}
    \pw [ n-2, \dots, k] \in \coll{n-1}{\bw_3},  & \textrm{ if } \pw \textrm{ ends with } n-1, \\
   \pw \bw_1 \in \coll{n}{\bw_1}, & \textrm{ otherwise.}
   \end{array}\right.
$$
One can check that the map $\rho$ is well defined.

Now one can see that $\rho$ is the two-sided inverse of $\varphi$.
In particular, if we restrict $\rho$ to the words whose prefixes end
with $n-1$, then we obtain the inverse for $\varphi_2$, while if we
restrict to the prefixes not ending in $n-1$, we have the inverse
for $\varphi_1$.

This establishes, for each $k$,
\[ |\coll{n}{\bw_2}| = |\coll{n}{\bw_1}|+|\coll{n-1}{\bw_3}| ,\]
which is the same identity as \eqref{sum rule} inductively.
This proves that $|\coll{n}{\bw_2}|=C(n,k)$ as desired.
\end{proof}

Let us recall the Catalan triangle polynomial introduced
in \cite[Definition 2.11]{LO16C}:

\begin{dfn} For $0 \le k \le n$, we define the {\em Catalan triangle polynomial} $\F_{n,k}(x)$ by
\begin{align} \label{eq: n,k Cat tr pol2}
\F_{n,k}(x) & = \sum_{s=0}^{k} C(n,s)x^{k-s}.
\end{align}
\end{dfn}

We need some special values of the polynomial $\F_{n,k}(x)$.

\begin{lemma}\cite[Corollary 2.9]{LO16C}
For $0 \le k <n$, we have
\begin{equation} \label{eq-fnk2}
\F_{n,k}(2) = \matr{n+1+k}{k}.
\end{equation}
\end{lemma}

In light of the above lemma, the numbers $\F_{n,k}(d)$, $d >2$, can be considered as a certain generalization of binomial coefficients. Interestingly, we need  $\F_{n,n-2}(d)$ to write a formula for the number of fully commutative elements in $G(d,1,n)$ in the following corollary. This is also the case for $G(d,r,n)$. See Corollaries \ref{cor-dd} and \ref{cor-dr}.

\begin{cor} \label{cor-end}
For $n\ge 3$, the number of fully commutative elements
of $G(d,1,n)$ is equal to
\begin{equation} \label{eqn-end} \sum_{k=0}^n C(n,k) \left | \PP(n, k) \right | = d(d-1)\F_{n,n-2}(d) +(2d-1)C_n -(d-1).\end{equation}
\end{cor}

In particular, when $d=2$, we recover the result of \cite{Stem96} on $B_n$-type using \eqref{eq-fnk2}:
$$\sum_{k=0}^n C(n,k) \left | \PP(n, k) \right | = 2\matr{2n-1}{n-2}+ 3C_n -1  = (n+2) C_n-1.  $$

\begin{proof}
The assertion follows from Proposition \ref{prop-d-1} and the definitions.
\end{proof}

\section{Packets in $G(d,r,n)$} \label{sec-gddn}

In this section, we assume that $ 1\le r \le d$ and $r|d$. Thus the family of groups $G(d,r,n)$ includes the case $G(d,d,n)$. The results will be presented for $G(d,d,n)$ first for simplicity, and then will be generalized to the case $G(d,r,n)$.

\medskip

 Recall that we consider $G(d,r,n)$ as a subgroups of $G(d,1,n)$ through the embedding $\tau$, and that  Lemma \ref{lem-tilde-r} describes the elements of $G(d,r,n)$. We define the {\em packets of $G(d,r,n)$} to be those of $G(d,1,n)$  which are contained in $G(d,r,n)$.

\begin{prop} \label{prop-d-d}
The size of the
packet $\PP(n,k)$ of $G(d,d,n)$ is
$$ \left | \PP(n,k) \right| =
   \left\{\begin{array}{cl}
   d^{n-k-2}(d-1) & \textrm{ if } 0\le k \le n-2, \\
   1 & \textrm{ if } k= n. \\
   \end{array}
   \right.
$$
Hence we have $ \sum_{k=0}^n  \left | \PP(n,k) \right|  =  d^{n-1}$.
\end{prop}

\begin{proof}
We use Lemma \ref{lem-tilde} to determine which packets of $G(d,1,n)$ in Definition \ref{dfn-lp} are contained in $G(d,d,n)$.
Clearly, the $(n,n-1)$-packet cannot occur in $G(d,d,n)$, and there is still only one collection in the $(n,n)$-packet.
For $0 \le k \le n-2$,
the suffixes $s^{(t)}_{nk}$ cannot appear in $G(d,d,n)$ and the number of suffixes of the form
$s^{(t_1)}_{n\, k+1}s^{(t_2)}_{nj_2}s^{(t_3)}_{nj_3}s^{(t_4)}_{nj_4}\cdots s^{(t_\ell)}_{nj_\ell}$ $(\ell \ge 2)$ that appear in $G(d,d,n)$
is $d^{n-k-2}(d-1)$.
\end{proof}

More generally, we have the following.

\begin{prop} \label{prop-d-r}
The size of the
packet $\PP(n,k)$ of $G(d,r,n)$ is
$$ \left | \PP(n,k) \right| =
   \left\{\begin{array}{cl}
   \dfrac{d^{n-1}}{r} (d-1) - \left(\dfrac{d}{r}-1 \right)  & \textrm{ if } k=0, \\
   \dfrac{d^{n-k-1}}{r} (d-1) & \textrm{ if } 1\le k \le n-2, \\
   2 \left(\dfrac{d}{r}-1 \right) & \textrm{ if } k= n-1, \\
   1 & \textrm{ if } k= n. \\
   \end{array}
   \right.
$$
Hence we have $ \sum_{k=0}^n  \left | \PP(n,k) \right|  = d^n/r$.
\end{prop}

\begin{proof}
We use Lemma \ref{lem-tilde-r} to determine which packets of $G(d,1,n)$ in Definition \ref{dfn-lp} are contained in $G(d,r,n)$.
Clearly, there is still only one collection in the $(n,n)$-packet. For the $(n,n-1)$-packet, each of $s^{(t)}_n$ and $s^{(t)}_{n,n-1}$ has $\frac d r-1$ possibilities to satisfy the conditions $t \equiv 0\  (\mathrm{mod}\ r)$ and $1\le t <d$.

As for the $(n,0)$-packet, we consider the expression
\[s^{(k_1)}_{n1}s^{(k_2)}_{n2}s^{(k_3)}_{n3}\cdots  s^{(k_{n-1})}_{n,n-1} s^{k_n}_{n}.\]
Then $k_1$ varies from $1$ to $d-1$ and $k_i$ $(2 \le i \le n-1)$ from $0$ to $d$ and then $k_n$ has $d/r$ choices, except the cases that $k_1 \equiv 0$ $(\mathrm{mod}\ r)$ and $k_2=k_3=\cdots =k_n=0$.
Thus the total number of collections in $\PP(n,0)$ is $\dfrac{d^{n-1}}{r} (d-1) - \left(\dfrac{d}{r}-1 \right)$.

The sizes of $(n,k)$-packets for $1 \le k \le n-2$ can be checked similarly.
\end{proof}

\begin{cor} \label{cor-dd}
The number of fully commutative elements in the group $G(d,d,n)$ is equal to
\begin{equation} \label{eqn-end2} \sum_{k=0}^n C(n,k) \left | \PP(n, k) \right | = (d-1)\F_{n,n-2}(d)+ C_n.\end{equation}
\end{cor}
In particular, when $d=2$, we obtain from \eqref{eq-fnk2}
\begin{equation} \label{eq:above} \sum_{k=0}^n C(n,k) \left | \PP(n, k) \right | = \matr{2n-1}{n-2}  + C_n  = \dfrac{n-1}{2} C_n + C_n  = \dfrac{n+1}{2} C_n.
\end{equation}

\begin{remark}
The number $\frac{n+1}{2} C_n$ in \eqref{eq:above}  is different from the number $\frac{n+3}{2} C_{n}-1$  of fully commutative elements of type $D_n$ considered in \cite{Stem98,FL14} without embedding $\iota$. Thus our definition of fully commutative elements of $G(2,2,n)$ is not equivalent to that of $D_n$ in \cite{Stem98,FL14}. See the first example in Section \ref{sec:ex} for more details.
\end{remark}

\begin{remark}

Let $c(x)=\dfrac {1 -\sqrt{1-4x}}{2x}$ be the generating function of the Catalan numbers $C_n$. The generating function of the numbers of fully commutative elements in $G(d,d,n)$ is given by
\[ \frac{1-(d-1)\, x \, c(x)}{1-d\, x \, c(x)} .\]
\end{remark}

\begin{cor} \label{cor-dr}
The number of fully commutative elements in the group $G(d,r,n)$ is equal to
\begin{equation} \label{eqn-endr}
\sum_{k=0}^n C(n,k) \left | \PP(n, k) \right | = \dfrac{d(d-1)}{r} \F_{n,n-2}(d) +\left( \dfrac{2d}{r} -1 \right)C_n - \left (\dfrac d r -1 \right ).
\end{equation}
\end{cor}

\section{Examples} \label{sec:ex}

(1) The group $G(2,2,n)$ is isomorphic to the Coxeter group of type $D_n$ which has its own definition of fully commutative elements without invoking the embedding $\iota$ into $G(2,1,n)$.
For example, the element $\ts_3 \ts_4 \ts_2 \ts_1 \in G(2,2,4)$ is fully commutative before being embedded into $G(2,1,4)$, but we have
\[ \iota ( \ts_3 \ts_4 \ts_2 \ts_1 ) = s_3 s_4 s_3 s_4 s_2 s_1 , \]
which is not fully commutative in $G(2,1,4)$. The number of fully commutative elements of $D_4$ (without embedding) is 48, whereas the number of fully commutative elements of $G(2,2,4)$ (after being embedded) is 35.

(2) The group $G(3,1,3)$ has 59 fully commutative elements. We list them below in packets and collections.

\begin{center}
\begin{tabular}{|c|c|c|}
 \hline
Packets   & \text{ Collections} & $C(3,k)$\\
   \hline
$\mathcal P(3,0)$ &
$ \begin{array}{llll} \mathbf{c}_{[3,2,1,3]}, & \mathbf{c}_{[3,2,1,3^2]}, & \mathbf{c}_{[3,2,1,3,2]}, & \mathbf{c}_{[3,2,1,3,2,3]}, \\
\mathbf{c}_{[3,2,1,3,2,3^2]},  & \mathbf{c}_{[3,2,1,3^2,2]},  & \mathbf{c}_{[3,2,1,3^2,2,3]}, &
\mathbf{c}_{[3,2,1,3^2,2,3^2]},\\
\mathbf{c}_{[3^2,2,1,3]},  & \mathbf{c}_{[3^2,2,1,3^2]}, &
\mathbf{c}_{[3^2,2,1,3,2]}, & \mathbf{c}_{[3^2,2,1,3,2,3]},\\
\mathbf{c}_{[3^2,2,1,3,2,3^2]}, & \mathbf{c}_{[3^2,2,1,3^2,2]},  &
\mathbf{c}_{[3,2,1,3^2,2,3]},  & \mathbf{c}_{[3,2,1,3^2,2,3^2]}
\end{array}$ & 1\\
   \hline
   $\mathcal P(3,1)$ &
$ \begin{array}{rl} \mathbf{c}_{[3,2,3]} \hskip - 0.2 cm & = \{{  s_3s_2s_3},{ s_1 s_3s_2s_3},{ s_2s_1 s_3s_2s_3}\} \\
 \mathbf{c}_{[3,2,3^2]} \hskip - 0.2 cm &=\{{  s_3s_2s_3^2},{ s_1 s_3s_2s_3^2},{ s_2s_1 s_3s_2s_3^2}\} \\
 \mathbf{c}_{[3^2,2,3]} \hskip - 0.2 cm &=\{{  s_3^2s_2s_3},{ s_1 s_3^2s_2s_3},{ s_2s_1 s_3^2s_2s_3}\} \\
 \mathbf{c}_{[3^2,2,3^2]} \hskip - 0.2 cm &=\{{  s_3^2s_2s_3^2},{ s_1 s_3^2s_2s_3^2},{ s_2s_1 s_3^2s_2s_3^2}\} \\
 \mathbf{c}_{[3,2,1]} \hskip - 0.2 cm &=\{{  s_3s_2s_1},{ s_2 s_3s_2s_1},{ s_1s_2 s_3s_2s_1}\} \\
 \mathbf{c}_{[3^2,2,1]}&=\{{  s_3^2s_2s_1},{ s_2 s_3^2s_2s_1},{ s_1s_2 s_3^2s_2s_1}\}
\end{array}$ & 3 \\
  \hline
  $\mathcal P(3,2)$  & 
 $ \begin{array}{rl} \mathbf{c}_{[3]}  \hskip - 0.2 cm & =\{{  s_3},{ s_1 s_3},{ s_2 s_3},s_1s_2{  s_3},s_2s_1{  s_3} \} \\
  \mathbf{c}_{[3^2]}  \hskip - 0.2 cm & =\{{  s_3^2},{ s_1 s_3^2},{ s_2 s_3^2},s_1s_2{  s_3^2},s_2s_1{  s_3^2} \}\\
    \mathbf{c}_{[3,2]}  \hskip - 0.2 cm & =\{{  s_3s_2},{ s_1 s_3s_2},{ s_2 s_3s_2},s_1s_2{  s_3s_2},s_2s_1{  s_3s_2} \}\\
    \mathbf{c}_{[3^2,2]}  \hskip - 0.2 cm & =\{{  s_3^2s_2},{ s_1 s_3^2s_2},{ s_2 s_3^2s_2},s_1s_2{  s_3^2s_2},s_2s_1{  s_3^2s_2} \}
  \end{array}$ & 5 \\
  \hline
  $\mathcal P(3,3)$ & 
 $ \mathbf{c}_{[]}=\{{ [~]},{ s_1},{ s_2},{ s_1s_2}, { s_2s_1} \}$ &5
 \\
   \hline
\end{tabular}
\end{center}

\medskip

(3) The set of reduced words for the group $G(3,3,3)$ is given in Appendix. To the canonical words, one applies the embedding $\iota: G(3,3,3) \hookrightarrow G(3,1,3)$ and sees that the group has 17 fully commutative elements. We list them all below, where we write $[i_1i_2 \dots i_k]$ for $\ts_{i_1}\ts_{i_2} \cdots \ts_{i_k} \in G(3,3,3)$ and $\la i_1i_2 \dots i_k \ra$ for $s_{i_1}s_{i_2} \cdots s_{i_k} \in G(3,1,3)$.

\begin{align*}
[ ]  & \longmapsto \la \ra, & [1] & \longmapsto \la 1 \ra,  & [2]  & \longmapsto \la 2 \ra,  \\
[3] & \longmapsto \la 3^223 \ra,&
[12]  & \longmapsto \la 12 \ra , &
[13]  & \longmapsto \la 13^223 \ra , \\
[21]  &\longmapsto \la 21 \ra , &
[31]  & \longmapsto \la 3^2231 \ra= \la 3^2213 \ra , &
[213] & \longmapsto \la 213^223 \ra, \\
[232]  &\longmapsto \la 23^2232 \ra = \la 323^2 \ra , &
[312]  &\longmapsto \la 3^22132 \ra, &
[1232] & \longmapsto \la 123^2232 \ra = \la 1323^2 \ra, \\
[2321] & \longmapsto  \la 3213^2 \ra, &
[12132] & \longmapsto \la 21 32 3^2 \ra, &
[13123]  &\longmapsto \la 3^2213^223^2 \ra,  \\
[23121]  &\longmapsto  \la 3213^22 \ra, &
[23213]  &\longmapsto  \la 321323 \ra.
\end{align*}

The $(3,0)$-packet has $6$ collections, each of which has only one element:
\[ \PP(3,0) = \left \{ \{\la 3^2213 \ra  \}, \{ \la 3^22132 \ra \}, \{\la 3213^2 \ra \},  \{\la 3^2213^223^2 \ra\} ,\{ \la 3213^22 \ra \},  \{ \la 321323 \ra \} \right \}.  \]
There are 2 collections in the $(3,1)$-packet, each of which has 3 elements:
\[ \PP(3,1) = \left \{ \{ \la 3^223 \ra, \la 13^223 \ra ,  \la 213^223 \ra\}, \{ \la 323^2 \ra , \la 1323^2 \ra,  \la 21 32 3^2 \ra \} \right \}.  \]
Recall that there is no $(3,2)$-packet.
There is only one collection in the $(3,3)$-packet and it has 5 elements:
\[ \PP(3,3) = \left \{ \{ \la \ra,  \la 1 \ra, \la 2 \ra, \la 12 \ra ,  \la 21 \ra \} \right \}. \]
All together we have
\[ 1 \times 6 + 3 \times 2 + 5 \times 1 =17. \]

\begin{remark}
Before taking the embedding $G(3,3,3) \hookrightarrow G(3,1,3)$, we may want to say that the element $\ts_2 \ts_3=[23]$ is fully commutative. After the embedding, we have \[ [23] \longmapsto \la 23^2 23 \ra  ,\] and the element is not fully commutative.
\end{remark}

\section{Appendix: reduced words for $G(3,3,3)$ without an embedding} \label{sec-app}

In this appendix, we write $[i_1i_2 \dots i_k]$ for $\ts_{i_1}\ts_{i_2} \cdots \ts_{i_k}$.

\begin{lemma} \label{lem-rere}

The following relations hold in $G(3,3,3)$:

\begin{equation*} \begin{array}{llll}
[31232]=[13123], & [32131]=[23213], & [213121]=[131213], & [213123]=[131231], \\

[213213]=[132132], &  [231213]=[123121], & [231231]=[123123], &   [232132]=[132131], \\

[312131]=[121312], & [312132]=[121321], &  [312312]=[123123],  & [321321]=[132132] .

\end{array}
\end{equation*}

\end{lemma}

\begin{proof}
All the relations are derived from the defining relations. For example, we have the defining relation $[313]=[131]$. Multiplying both sides by $\ts_2 \ts_3$ from the right, we obtain
\[ [31323]=[31232]=[13123] ,\] where we use another defining relation $[323]=[232]$. Thus we obtain $[31232]=[13123]$.
\end{proof}

\begin{prop} \label{prop-rw 333}
A set of reduced words for $G(3,3,3)$ is given by
\[ \begin{array}{lllllllll}

[\ ],      \\

[1], & [2], & [3],

\\

[12], &  [13], &  [21], &  [23], &  [31], &  [32]

\\

[121], &  [123], &  [131], &  [132], &  [213], &  [231], &  [232], &  [312], &  [321],

\\

[1213], &  [1231], &  [1232], & [1312], &  [1321], &  [2131],

\\

[2132], &  [2312], &  [2321], &  [3121],  &   [3123], &  [3213],

\\

[12131], &  [12132], & [12312], &  [12321], & [13121], &  [13123], &  [13213], &  [21312],

\\

[21321], & [23121], & [23123], &  [23213], &  [31213], &   [31231],  &    [32132],

\\

[121312], &  [121321],  &  [123121], &  [123123],  &  [123213],&  [131213], &  [131231],  &  [132132] .
\end{array}
\]

\end{prop}

\begin{proof}
We set an ordering $1<2<3$ on the alphabet $I=\{1,2,3 \}$ and use the degree-lexicographic ordering on the set of words on $I$. Then one can see that the words in the list above do not contain as a subword any of the leading words of the defining relations for $G(3,3,3)$ and of the relations of Lemma \ref{lem-rere}. Further it can be checked that the list has all the words with this property.
The number of words in the list is $54$, which is exactly the order of $G(3,3,3)$. Thus it follows from the \GS basis theory that the list is a set of reduced words for $G(3,3,3)$.
\end{proof}

\end{document}